\documentclass{amsart}
\usepackage{amsmath}%
\usepackage{amsfonts}%
\usepackage{amssymb}%
\usepackage{graphicx}
\usepackage{mathtools, float,overpic}

\DeclareMathOperator{\sys}{sys}
\newcommand{\p}{\mathcal{P}}

\newcommand{\R}{\mathbb{R}}
\newcommand{\Z}{\mathbb{Z}}

\renewcommand{\S}{\mathcal{S}}
\newcommand{\T}{\mathcal{T}}

\newtheorem{theorem}{Theorem}[section]
\theoremstyle{plain}
\newtheorem{lem}[theorem]{Lemma}

\newtheorem{cor}[theorem]{Corollary}

\numberwithin{equation}{subsection}
\theoremstyle{definition}

\theoremstyle{definition}
\newtheorem{rem}[theorem]{Remark}
\newtheorem{ques}[theorem]{Question}

\title[Building metrics and applications to lifting]{Building hyperbolic metrics suited to closed curves and applications to lifting simply}
\author{Tarik Aougab, Jonah Gaster, Priyam Patel, Jenya Sapir}

\begin{document}

\address{
\newline Department of Mathematics\\ Brown University\\
\newline \it{E-mail address}: \tt{tarik\_aougab@brown.edu}
}

\address{
\newline Department of Mathematics\\ Boston College\\
\newline \it{E-mail address}: \tt{gaster@bc.edu}
}

\address{
\newline Department of Mathematics\\ Purdue University\\
\newline \it{E-mail address}: \tt{patel376@purdue.edu}
}

\address{
\newline Department of Mathematics\\ University of Illinois\\
\newline \it{E-mail address}: \tt{jsapir2@illinois.edu}
}

\date{March 18, 2016}

\begin{abstract} Let $\gamma$ be an essential closed curve with at most $k$ self-intersections on a surface $\S$ with negative Euler characteristic. In this paper, we construct a hyperbolic metric $\rho$ for which $\gamma$ has length at most $M \cdot \sqrt{k}$, where $M$ is a constant depending only on the topology of $\S$. Moreover, the injectivity radius of $\rho$ is at least $1/(2\sqrt{k})$. This yields linear upper bounds in terms of self-intersection number on the minimum degree of a cover to which $\gamma$ lifts as a simple closed curve (i.e.~lifts simply). We also show that if $\gamma$ is a closed curve with length at most $ L$ on a cusped hyperbolic surface $\S$, then there exists a cover of $\S$ of degree at most $N \cdot L \cdot e^{L/2}$ to which $\gamma$ lifts simply, for $N$ depending only on the topology of $\S$. 
\end{abstract}

\maketitle

\section{Introduction}

Let $\S$ be an orientable surface of negative Euler characteristic and of finite type. It is a consequence of a well-known result of P.~Scott \cite{Scott} that every closed curve $\gamma$ on $\S$ lifts to a simple closed curve in a finite degree cover $\widetilde{\S}_\gamma$ of $\S$; we say that $\gamma$ \emph{lifts simply} to $\widetilde{\S}_\gamma$. For a closed curve $\gamma$ on $\S$, let $\deg(\gamma)$ be the smallest degree of a cover where $\gamma$ lifts simply. Recently, there has been an effort to prove effective versions of Scott's result by obtaining bounds on $\deg(\gamma)$ \cite{gaster1, Gupta-Kapovich, Patel}. One of the main goals of this paper is to study the dependence of $\deg(\gamma)$ on two other notions of complexity; the geometric self-intersection number $i(\gamma, \gamma)$ of the curve, and, fixing a complete hyperbolic metric $\rho$ on $\S$, the length $\ell_\rho(\gamma)$.

We are interested in the following two functions, defined in \cite{Gupta-Kapovich}: 
\begin{align*}
f_\rho(L) &= \max \left \{\deg(\gamma) : \ell_\rho(\gamma) \leq L \right \} ~, \text{ and} \\
f_\S(k) &= \max \left \{\deg(\gamma) : i(\gamma, \gamma) \leq k \right \}~.
\end{align*}

\noindent In order to study these functions, we consider the relationship between the length and self-intersection number of closed curves on hyperbolic surfaces. This has been studied extensively in a variety of different contexts  (see e.g. \cite{Basmajian, Chas-Phillips1, Chas-Phillips2, Hempel-lengths, lalley, Nakanishi, Rivin}).

Our main contribution to this story is Theorem \ref{thm: shortest metric}, in which we produce a metric $\rho$ for each curve $\gamma$ of self-intersection at most $k$ on $\S$ so that $\gamma$ has $\rho$-length at most a constant times $\sqrt{k}$. Theorem \ref{thm: shortest metric} implies upper bounds for curve degrees (see Theorem \ref{thm: lifts simply in n}), and lower bounds for $f_\S$ are provided by \cite[Thm.~1.2]{gaster1}. 

Recall that $f = \Theta(g)$ means that there are constants $c_1$ and $c_2$ so that $c_1 g(x) \le f(x) \le c_2 g(x)$ for all $x$. We obtain:

\begin{cor}[Linear growth of $f_\S$]
\label{cor: linear characterization}
We have $f_\S(k)= \Theta(k)$, with constants that depend only on the topology of $\S$.
\end{cor}

In Theorem \ref{thm:cusped lift simply}, we analyze the relationship of the length of a curve on a cusped hyperbolic surface to the length of the corresponding curve when the cusps have been `opened' to geodesic boundary. Together with \cite[Thm.~1.1]{Patel} we obtain upper bounds for $f_\rho$ when $(\S,\rho)$ has cusps, and lower bounds are provided by \cite[Thm.~1.4]{gaster1}. We obtain:

\begin{cor}[Exponential growth of $f_\rho$ in cusped case]
\label{cor: exp characterization}
Suppose that $(\S, \rho)$ is a hyperbolic surface with at least one cusp. Then there exists  $L_0(\rho) > 0$ and constants $C_1$ and $C_2$ depending only on the topology of $\S$ so that for each $L \geq L_0$, 
\[ C_1 \cdot e^\frac{L}{2} \leq f_\rho(L) \leq C_2 \cdot L \cdot e^\frac{L}{2}~. \]
\end{cor}
 
\noindent A natural question left open in our work is which of these two bounds is sharp.

Finally, Theorem \ref{thm: shortest metric} informs the quantitative understanding of infima of length functions associated to curves on $\S$. Note that there are only finitely many mapping class group orbits of closed curves with $k$ self-intersections on $\S$. Each such orbit determines a length function on the moduli space of hyperbolic structures on $\S$, and this length function has an infimum. Let $m_k(\S)$ and $\overline{m}_k(\S)$ indicate the minimum and maximum, respectively, of the infima among mapping class group orbits of curves of self-intersection $k$.

In \cite[Cor.~1.4]{Basmajian}, Basmajian shows that $m_k(\S)$ is $\Theta(\log k)$, and it's a consequence of the construction in the first part of \cite[Cor.~1.3]{Basmajian} that $\overline{m}_k(\S)$ grows at least like $\sqrt{k}$. Theorem \ref{thm: shortest metric} immediately implies:

\begin{cor} \label{max infimum} We have $\overline{m}_k(\S)= \Theta( \sqrt{k})$.
\end{cor}

\subsection{Self-intersection and lengths of curves}
In \cite[Thm.~1.1]{Basmajian}, Basmajian shows that for any closed curve $\gamma$ of self-intersection $k$ on a compact hyperbolic surface $(\S, \rho)$ we have $c_\rho \cdot \sqrt{k} \leq \ell_\rho(\gamma)$, where $c_\rho$ is a constant depending on the metric $\rho$. The following provides a complementary upper bound, answering the second question asked in \cite{gaster2}:

\begin{theorem}\label{thm: shortest metric} Let $\gamma$ be a closed curve with $k$ self-intersections on $\S$. Then there exists a hyperbolic metric $\rho$ on $\S$ satisfying 
 \[
 \ell_\rho(\gamma) \leq C_3 \cdot \sqrt k~,
 \]
 where $C_3$ is a constant depending only on $\S$. Furthermore, the injectivity radius of $\rho$ is at least $1/(2\sqrt k)$.

\end{theorem}

The proof relies on recent work of the fourth author \cite{Sapir}, in which a careful accounting of mapping class group orbits of curves is obtained. Briefly, we choose a pants decomposition (quoting \cite[Prop.~1.4]{Sapir}) so that $\gamma$ has controlled intersection with the pants cuffs. With such a pants decomposition in hand, we obtain a natural decomposition of the curve into arcs (see Section \ref{sec: short metric prelim}), and careful choices for lengths of pants cuffs are made so that the sum of the lengths of the constituent arcs can be controlled. The systole estimate arises naturally.

\begin{rem} \label{by hand} In many concrete cases, the method of proof of Theorem \ref{thm: shortest metric} produces a metric with an improved systole length. This allows the applications to degree (see Theorem \ref{thm: lifts simply in n}) to be strengthened. For further discussion, see Section \ref{sec:Questions and remarks} and Remark \ref{rem:Thick part condition}.
\end{rem}

\subsection{Bounds on $f_\rho$ and $f_\S$ for surfaces without cusps}
In \cite{Scott}, Scott explicitly constructs a regular, right-angled pentagonal tiling for any surface of negative Euler characteristic. We now fix the notation $\rho_0$ for the hyperbolic metric on $\S$ arising from this tiling. In \cite{Patel}, the third author used $\rho_0$ to obtain an upper bound on $f_\rho(L)$ that is linear in $L$ in the case where $(\S, \rho)$ does not have any cusps. In particular, for $\rho_0$ one has $f_{\rho_0}(L) < 16.2 \cdot L$ \cite[Thm~1.1]{Patel}. Gupta-Kapovich then gave a lower bound on $f_\rho(L)$ \cite{Gupta-Kapovich} that was recently improved upon by the second author. Combining the results of the second author \cite{gaster1} and third author \cite{Patel} gives the following description of $f_\rho(L)$ found in \cite[pp.~3]{gaster1}:

\begin{cor}[Gaster, Patel]\label{cor:linear characterization}
If $(\S, \rho)$ is a hyperbolic surface without cusps, there exists a constant $C = C(\S) > 0$ and $L_0 = L_0(\rho) > 0$ such that for $L \geq L_0$, 

\[ \left(\frac{1}{C \cdot \sys(\rho)}\right) L \leq f_\rho(L) \leq \left ( \frac{C}{\sys(\rho)}\right) L ~, \]

\noindent where $\sys(\rho)$ is the systole of $(\S, \rho)$. 

\end{cor}

As noted by Gupta-Kapovich, a lower bound on $f_\rho(L)$ gives a lower bound on $f_{\S}(k)$ by appealing to a theorem of Basmajian \cite[Thm.~1.1]{Basmajian}. In \cite{gaster1}, the second author also improved the lower bound on $f_\S(k)$ obtained by Gupta-Kapovich without appealing to Basmajian's work relating self-intersection number and length of a closed curve, but rather by analyzing a specific family of curves (also analyzed by Basmajian \cite[Prop.~4.2]{Basmajian}). In particular, \cite[Thm.~1.2]{gaster1} shows that $f_\S(k)\ge k+1$.

Our goal is to produce a complementary upper bound on $f_\S(k)$ that is linear in $k$. Unfortunately, the upper bound on $f_\rho(L)$ from \cite{Patel} does not immediately yield an upper bound on $f_{\S}(k)$ as there exist arbitrarily long closed curves on a hyperbolic surface with few (or no) self-intersections. This issue is addressed by Theorem \ref{thm: shortest metric} above, in which a metric $\rho$ is produced that is well suited to a comparison between $\ell_\rho(\gamma)$ and the self-intersection $i(\gamma,\gamma)$. 

As in \cite[Lem.~7]{gaster1}, the change in length arising from the passage from $\rho$ back to $\rho_0$ can be controlled via the work of Lenzhen-Rafi-Tao \cite{L-R-T} on optimal Lipschitz constants between hyperbolic surfaces, with the caveat that control on the systole length of $\rho$ is required. This demonstrates the use of the systole estimate in the conclusion of Theorem \ref{thm: shortest metric}.

We note that an argument in the style of \cite{Hempel-residual} could result in a crude upper bound for $f_\S(k)$ that is exponential in $k$. Our theorem represents a vast improvement on this naive bound. 

\begin{theorem}\label{thm: lifts simply in n}
There is a constant $C_4$ depending only on $\S$ such that 
$f_\S(k) \leq C_4 \cdot k$.
\end{theorem}

\begin{proof}
Suppose $\gamma$ is a closed curve on $\S$ with at most $k$ self-intersections. If necessary, replace any punctures of $\S$ with boundary components, and note that $\deg(\gamma)$ is unaffected by this change. By Theorem \ref{thm: shortest metric}, there is a $1/(2\sqrt{k})$-thick metric $\rho$ on $\S$ so that $\gamma$ has $\rho$-length at most $C_3 \cdot \sqrt{k}$. Translating \cite[Thm.~E]{L-R-T} as in \cite[Lem.~7]{gaster1}, there is a constant $C_0$ depending only on $\rho_0$ so that there is a Lipschitz map from $\rho$ to $\rho_0$ with Lipschitz constant $C_0/\sys(\rho)$. 

Thus we have 
\[ \ell_{\rho_0}(\gamma) \le \frac{C_0}{\sys(\rho)} \cdot \ell_\rho(\gamma) \le 2C_0\cdot C_3 \cdot k~,
\]
and \cite[Thm~1.1]{Patel} completes the proof.
\end{proof}

In fact, it is possible to achieve the conclusion of Theorem \ref{thm: lifts simply in n} without resorting to the kind of length control obtained in Theorem \ref{thm: shortest metric}, an argument communicated to us by Justin Malestein \cite{Malestein-private-communication}. This argument uses the finiteness of the number of mapping class group orbits of incompressibly embedded roses on $\S$, and finds a curve in the mapping class group orbit of $\gamma$ whose $\rho_0$-length can be controlled so that one may invoke \cite[Thm.~1.1]{Patel} directly. On the other hand, in addition to its independent interest, Theorem \ref{thm: shortest metric} includes fine control on the geometry of $\gamma$ on $\S$, and this control can be used to deduce more information about degrees of curves. 

For instance, our proof of Theorem \ref{thm: lifts simply in n} also yields control over the dependence of degree on the Euler characteristic of the underlying surface $\S$. Precisely, we show:

\begin{theorem} There exists a function $F: \mathbb{N}^{2} \rightarrow \mathbb{N}$ and a universal constant $R$ satisfying:
\begin{enumerate}
\item For any orientable surface $\S$ of finite type, $f_{\S}(k) \leq F(|\chi(\S)|, k)$;
\item For each $a \in \mathbb{N}$, $F(a, y)= O \left( y \right)$;
\item For each $b \in \mathbb{N}$, $F(x, b) =O \left( R^{x} \right)$.
\end{enumerate}
\end{theorem}

\begin{proof}   The constant $C_3$ in the statement of Theorem \ref{thm: shortest metric} grows at most exponentially in the genus of the underlying surface $\S$; indeed, examining each of the terms in the inequality at the beginning of the proof of Theorem \ref{thm: shortest metric} carefully, we see that $C_3$ depends polynomially on the constant $C$ in the statement of Lemma \ref{self-intersection inequalities}, which in turn depends quadratically on the constant $D$ in the statement of Lemma \ref{Jenya's pants}. This last statement and the proof of Lemmas \ref{Jenya's pants} and \ref{self-intersection inequalities} in \cite{Sapir} demonstrate an exponential upper bound for the growth of $C_3$ as a function of $|\chi(S)|$. 

Therefore, for $\gamma$ a closed curve on $\S$ with $i(\gamma, \gamma) \leq k$, Theorem \ref{thm: shortest metric} yields a hyperbolic metric with injectivity radius at least $1/(2\sqrt{k})$, for which $\gamma$ has length $\leq (R')^{|\chi(S)|} \cdot \sqrt{k}$ for some universal constant $R'>0$. Then by \cite{RafiTao}, the diameter of the $1/(2\sqrt{k})$-thick part of the Moduli space (equipped with the Teichm{\"u}ller metric) is at most $R'' \cdot \sqrt{k} \cdot \log(|\chi(S)|)$, for some universal $R''$. As the Teichm{\"u}ller metric bounds the Lipschitz metric from above, it follows that $\ell_{\rho_{0}}(\gamma)$ is bounded above by 
\[ (R')^{|\chi(S)|} \sqrt{k} \cdot |\chi(S)|^{R'' \sqrt{k}}=: F_{1}(|\chi(\S)|, k). \]
Then choose $R \gg \max(R',R'')$, and again appealing to \cite[Thm~1.1]{Patel}, we have that the degree of $\gamma$ is at most $16.2 \cdot F_{1}$. We conclude by letting $F(|\chi(\S)|, k):= \min(16.2 \cdot F_{1}, C_4 \cdot k)$, where $C_4$ is the constant from Theorem~\ref{thm: lifts simply in n}. 

\end{proof}

\subsection{Bounds on $f_\rho$ for surfaces with cusps}
In our study of $f_\rho$ when $\rho$ has cusps, it will be useful to compare the length of a curve on the cusped surface to its length when the cusps have been replaced with geodesic boundary components. Note that $\rho_0$ below could be any hyperbolic metric on $\S$ without cusps, though in the applications that follow we are only concerned with Scott's pentagonally-tiled metric. 

\begin{theorem}\label{thm:difference in lengths}
Let $(\S, \rho)$ be a hyperbolic surface with at least one cusp and $\rho_0$ a metric on $\S$ without cusps. Then there exists an $L_0 > 0$ and a constant $C_0=C_0(\rho_0)$ depending on $\rho_0$, such for any $L \geq L_0$ and any closed curve $\gamma$ on $\S$ with $\ell_\rho(\gamma) \leq L$ we have $\ell_{\rho_0}(\gamma) \leq C_0 \cdot L \cdot e^{\frac{L}{2}}~.$ 
\end{theorem}

In \cite{gaster1}, the second author obtains an interesting lower bound on $f_\rho(L)$ in the case where $(\S, \rho)$ is a hyperbolic surface with at least one cusp, which in particular shows that the upper bound on $f_\rho(L)$ in Corollary \ref{cor:linear characterization} cannot hold in the cusped setting. Specifically, it is shown that when $(\S, \rho)$ is a hyperbolic surface with at least one cusp, $f_\rho(L)$ is roughly bounded from below by $e^{L/2}$. It is important to note $\deg(\gamma)$ is independent of the metric on $\S$. The discrepancy between the linear upper bound for $f_\rho(L)$ in the compact case and the exponential lower bound in the cusped case arises from the change in the length of a closed curve on a compact hyperbolic surface as we let the lengths of boundary curves go to zero. As an application of the length estimate in Theorem \ref{thm:difference in lengths} we obtain the following upper bound for $f_\rho(L)$ in the cusped case: 

\begin{theorem}\label{thm:cusped lift simply}
Let $(\S, \rho)$ be a hyperbolic surface with at least one cusp. Then there exists $L_0 > 0$ such that for any $L \geq L_0$ we have 
\[
f_\rho(L) \leq C_2 \cdot L \cdot e^{L/2}
\]
where $C_2 = 16.2 \cdot C_0$ and $C_0$ is a constant that depends only on $\S$. 
\end{theorem}

\begin{proof}
Given $\gamma$ on $(\S, \rho)$ with $\ell_\rho(\gamma) \leq L$, Theorem \ref{thm:difference in lengths} gives us that $\ell_{\rho_0}(\gamma) \leq C_0 \cdot L \cdot e^{L/2}$, where $C_0$ depends on the metric $\rho_0$. Since $\rho_0$ is fixed once and for all as Scott's pentagonally-tiled metric on $\S$, the constant $C_0$ depends only on the topology of $\S$. By Theorem 1.1 of \cite{Patel}, we have 
\[ \deg(\gamma) \leq 16.2 \cdot \ell_{\rho_0}(\gamma) \leq 16.2 \cdot C_0 \cdot L  \cdot e^{L/2}~,\] 
which completes the proof. 
\end{proof}

We prove Theorem \ref{thm:difference in lengths} in Section \ref{sec:cusped}. We introduce our combinatorial terminology in Section \ref{sec: short metric prelim} and prove some related technical estimates in Section \ref{sec: inequalities}. Theorem \ref{thm: shortest metric} is proved in Section \ref{sec:short geodesics}, after which we conclude with some remarks and questions.

\section{Proof of Theorem \ref{thm:difference in lengths}}\label{sec:cusped}

\begin{proof}[Proof of Theorem \ref{thm:difference in lengths}] In what follows we use the shorthand $\S$ and $\S_0$ for the hyperbolic surfaces $(\S, \rho)$ and $(\S, \rho_0)$, respectively. When these surfaces have non-compact, infinite volume ends, we replace such ends with totally geodesic boundary components. The result is unaffected by this change since the geodesic representative of a closed curve $\gamma$ on $\S$ never penetrates a non-compact end. 

Every cusp $c$ of $\S$ has an embedded neighborhood $N_{c}$ topologically equivalent to a cylinder, bounded by a horocycle $H_{c}$ of length $2$. Let $H_{0}$ be an embedded horocycle (parameterized by a parameter $t$ with respect to arclength) on the interior of $N_{c}$, whose length we denote by $l(H_0)$. Then $N_{c}$ can be parameterized by coordinates 

\begin{equation} \label{punc}
(\rho, t) \in (-\infty, d(0)] \times S^{1}~,
\end{equation}

\noindent where $d(0)$ is the distance between $H_{0}$ and $H_{c}$. With respect to this parameterization the hyperbolic metric can be expressed as 

\[ ds^{2}= d\rho^{2} + l^{2}(H_{0})e^{2\rho}dt^{2}~. \]

Concretely, if $x = (\rho, t) \in N_{c}$, then $\rho$ is the signed distance between $x$ and $H_{0}$ (where we define the distance between $x$ and $H_{0}$ to be negative if and only if $H_{0}$ does not separate $x$ from the cusp), and $t$ is the unique time so that the perpendicular from $x$ to $H_{0}$ intersects $H_{0}$ at time $t$.

Let $\eta$ be a totally geodesic boundary component of a hyperbolic surface. By the collar lemma, there exists an embedded half-collar neighborhood $N_{\eta}$ topologically equivalent to a cylinder, and metrically equivalent to a cylinder coordinatized by

\begin{equation} \label{comp}
(\rho, t) \in \left[0,  \sinh^{-1}\left(\frac{1}{\sinh\left(\frac{l(\eta)}{2} \right) } \right) \right] \times S^{1} ~, 
\end{equation}

\noindent with metric tensor
\[ ds^{2}= d\rho^{2} + l^{2}(\eta)\cosh^{2}(\rho)dt^{2}~. \]

Now, let $\gamma$ have length $\leq L$ on $\S$. Let $\S'$ denote the surface obtained from $\S$ by replacing each cusp with a totally geodesic boundary component of length $e^{-L/2}$. In other words, after specifying $\S= \left\{ (l_{1}, \theta_{1}),..., (l_{i}, \theta_{i}) \right\}$ in Fenchel-Nielsen coordinates with respect to some choice of pants decomposition, $l_{j}= 0$ for at least one $j$ (by assumption that $\S$ has at least one cusp). Then $\S'$ is the surface with the same Fenchel-Nielsen parameters after replacing each length $0$ coordinate with a length coordinate equal to $e^{-L/2}$.

 We will show that given any $\epsilon>0$, there exists $L$ sufficiently large so that $\gamma$ admits a representative of length $\leq L+\epsilon$ on $\S'$. Assuming this, we complete the argument as follows: for sufficiently large $L$, $e^{-L/2}$ will be the length of the shortest homotopically non-trivial curve on $\S'$, and therefore there is a Lipschitz map with Lipschitz constant bounded by $C_0 \cdot e^{L/2}$, between $\S'$ and the surface $\S_0$, where $C_0$ is a topological constant depending on $\S$ (see \cite{L-R-T} and \cite[Lem.~7]{gaster1} for details and for a description of the constant $C_0$). It then follows that the length of $\gamma$ is at most $C_0 \cdot (L+\epsilon) \cdot e^{L/2}$ on $\S_0$, and replacing $C_0$ by a slightly larger constant $C_1$, we have that the length of $\gamma$ is at most $C_1 \cdot L \cdot e^{L/2}$ on $\S_0$.

Thus, we have reduced the problem to showing that for $L$ sufficiently large, the length of $\gamma$ on $\S'$ is at most $L+\epsilon$. To avoid confusion, henceforth we will refer to the geodesic representative of $\gamma$ on $\S'$ as $\gamma'$. Note that on $\S$, $\gamma$ can penetrate at most $L/2$ into $N_{c}$ for any cusp $c$ of $\S$. For $a>1$, the horocycle of length $2/a$ is located at a signed distance of $\log(a)$ from $H_{c}$. Thus the horocycle of length $e^{-L/2}$, which we denote by $H_{L}$, is at a distance of $\log(2)+ L/2$ from $H_{c}$, and thus it separates $\gamma$ from the cusp $c$. We make a direct comparison between the metric on $N_{c}$ with respect to the horocycle $H_{L}$, and the metric on $\S'$ when viewed from a boundary component $\eta$ of length $e^{-L/2}$: since $\cosh^{2}(r) < e^{2r}$ for any $r>0$, it follows that the metric tensor on the half-collar neighborhood $N_{\eta}$ is bounded above by the metric tensor on the subset of $N_{c}$ bounded between $H_{c}$ and $H_{L}$. 

Let $W_{c} \subset \mathbb{H}^{2}$ denote a fundamental domain for $\S \setminus \bigcup_{c} N_{c}$, where the union is taken over all cusps $c$ of $\S$; $W_{c}$ is a compact region of $\mathbb{H}^{2}$ bounded by finitely many geodesic segments and finitely many segments of horocycles. Let $W_{\eta}$ denote a fundamental region for $\S' \setminus \bigcup_{\eta} N_{\eta}$, where the union is taken over all boundary components of $\S'$ coming from cusps of $\S$. Then as $L \rightarrow \infty$, $W_{\eta}$ converges in the Gromov-Hausdorff sense to $W_{c}$. Thus, for $L$ sufficiently large, the metric on the subsurface of $\S$ corresponding to the quotient of $W_{c}$ by $\pi_{1}(\S)$, is arbitrarily close to the metric on the subsurface of $\S'$ corresponding to the quotient of $W_{\eta}$ by $\pi_{1}(\S')$. 

Now, $\gamma$ decomposes as a union of geodesic arcs with endpoints on $H_{c}$, and similarly $\gamma'$ can be written as a concatenation of geodesic arcs with endpoints on $\partial N_{\eta}$. The interior of each arc of $\gamma$ (resp. $\gamma'$) lies completely within $N_{c}$ (resp. $N_{\eta}$) or completely within $W_{c}/\pi_{1}(\S)$ (resp. $W_{\eta}/\pi_{1}(\S')$). Thus we can compare the length of $\gamma$ to that of $\gamma'$ arc-wise: corresponding arcs in the complement of $N_{c}$ and $N_{\eta}$ are uniformly comparable within $\epsilon'$ (chosen smaller than $\epsilon/n$, where $n$ is the total number of arcs) by geometric convergence of $W_{\eta}$ to $W_{c}$, and arcs of $\gamma'$ within $N_{\eta}$ have length bounded above by the corresponding arc in $N_{c}$ since the metric tensor on $N_{\eta}$ is bounded above by the metric on $N_{c}$.  By this we mean that if $p, p'$ are two paths, one in $N_{c}$ and the other in $N_{\eta}$ respectively, so that $p$ and $p'$ have the same parameterization in coordinates, the length of $p'$ is at most that of $p$.  

Expounding on this last point,  geometric convergence of $W_{\eta}$ to $W_{c}$ implies that by choosing $L$ sufficiently large, the coordinates (in the circle $S^{1}$) of the endpoints of any arc in $\S$ on $H_{c}$ can be chosen to be arbitrarily close to the coordinates of the endpoints of the corresponding arc in $\S'$. Therefore, for any arc $\lambda$ in $N_{\eta}$, we can choose a parameterization in the coordinates of $(\ref{comp})$ that is arbitrarily close to the parameterization (in the coordinates of $(\ref{punc})$) of the corresponding arc in $N_{c}$; note that this parametrization may not yield the geodesic representative of $\lambda$ on $N_\eta$, and note also that we are using the fact that 
\[ \sinh^{-1}\left(\frac{1}{\sinh\left( \exp(-L/2) \right) } \right)  >L/2~. \]
Since the coordinates of $\lambda$ agrees with those of its counterpart on $\S$ (up to arbitrarily small precision), and the metric tensor on $N_{c}$ is an upper bound for the metric on $N_{\eta}$, the length of $\lambda$ is bounded above by the length of its counterpart on $\S$.

\end{proof}

\section{Combinatorial models for closed curves}
\label{sec: short metric prelim}
We begin with the combinatorial decomposition of a curve necessary for the proof of Theorem \ref{thm: shortest metric}. For convenience we will assume that $\S$ is compact in this construction; after proving Theorem \ref{thm: shortest metric} in that setting, we will deduce it in the case that $\S$ has cusps or non-compact ends. In everything that follows, the curve $\gamma$ on $\S$ of self-intersection $k$ is fixed. The first step is to choose a pants decomposition on $\S$ that is `well-balanced' with respect to $\gamma$, a construction contained in \cite{Sapir}.

\begin{lem}{\cite[Proposition 4.3]{Sapir}}
\label{Jenya's pants}
{There is a constant $D=D(\S)$ and a pants decomposition $\p=\p(\gamma)$ of $\S$ so that $\gamma$ intersects each cuff at most $D\sqrt{k}$ times.}
\end{lem}

The pants decomposition $\p$ is fixed in everything that follows.

\subsection{Realizing $\gamma$ as a cycle in a graph on $\S$} We fix a graph in each pair of pants so that each pants cuff is homotopic to a cycle of length two, and so that the complement of the graph consists of two hexagons, as shown in Figure \ref{fig:pantshexagons}. 
We will call the edges of this graph that connect distinct pants cuffs \emph{seam edges}, and the other edges \emph{boundary edges}. We glue together pants according to the combinatorics of the pants decomposition obtained in Lemma \ref{Jenya's pants}, gluing boundary edges to boundary edges and seam endpoints to seam endpoints. This way we obtain an embedded graph $G$ on $\S$. Note that its vertices have valence four on the interior of $\S$ and valence three on the boundary.

\begin{figure}[h!]
 \centering 
 \includegraphics{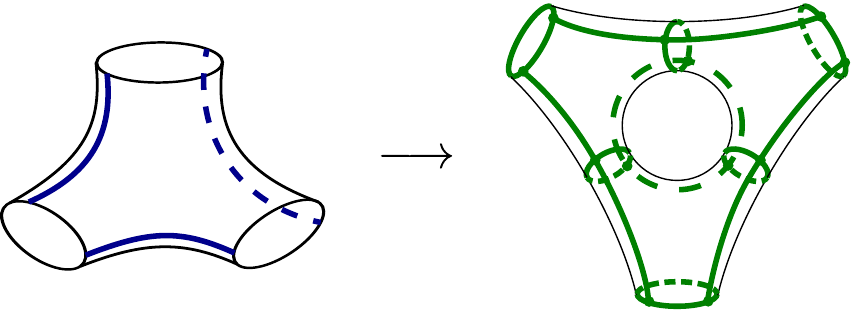}
 \caption{Pairs of pants cut into two hexagons along seam edges, as on the left, are glued together into the graph $G$ on $\S$, on the right.}
 \label{fig:pantshexagons}
\end{figure}

The pants decomposition $\p$ from Lemma \ref{Jenya's pants} makes the construction of hyperbolic structures on $\S$ possible `by hand':  A hyperbolic structure on a pair of pants with geodesic boundary is determined up to isometry by the lengths of the cuff curves, and one may glue the pants along their geodesic boundaries by matching seams to seams appropriately according to the combinatorics of $\p$. The conclusion of Theorem \ref{thm: shortest metric} is achieved by describing careful choices of length coordinates for the pants decomposition obtained in Lemma \ref{Jenya's pants}, and by making judicious choices of twisting $\gamma$ around cuffs to control its length. We note that with respect to natural choices for the Fenchel-Nielsen cordinates of the Teichm\"{u}ller space of $\S$, any metric we construct will have `twist coordinates' equal to half-integers, which we refer to below as a \emph{convenient} metric on $\S$.

Every closed curve on $\S$ is homotopic to a cycle of $G$. There is natural ambiguity in this description of the curve as a cycle, which we eliminate as follows. Choose a convenient hyperbolic structure on $\S$, and realize the graph $G$ so that each of the seam edges is the simple geodesic arc orthogonal to a pair of pants cuffs, and so that each boundary edge is half of the geodesic representative of a pants cuff. 

For any closed curve $\eta$, realize it as a closed geodesic. Following the construction in \cite[Lemma 2.2]{Sapir}, we obtain a cycle $c(\eta)$ freely homotopic to $\eta$ (Figure \ref{fig:SurfaceProjection}). Loosely speaking, the cycle $c(\eta)$ is formed by pushing $\eta$ onto the graph $G$ so that it has the fewest visits to pants curves among all cycles in $G$ homotopic to $\eta$. (Note that in \cite{Sapir}, $c(\eta)$ is refered to as a word $w_\Pi(\eta)$ whose letters are edges in $G$.) Though the cycle $c(\eta)$ constructed in \cite[Lemma 2.2]{Sapir} may depend on the chosen hyperbolic structure on $\S$, it is well-defined as a cycle of $G$ embedded on $\S$, so we may now alter the hyperbolic structure at will.

\begin{figure}[h!]
 \centering 
 \includegraphics{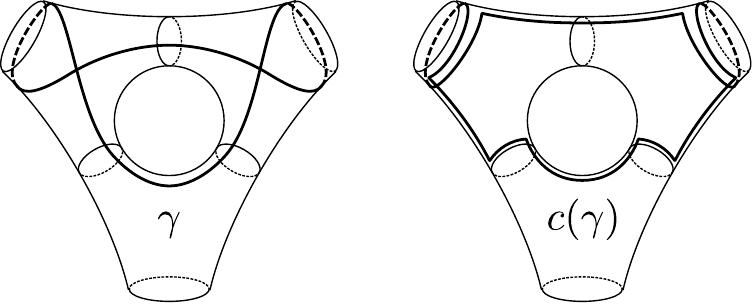}
 \caption{The curve $\gamma$ and the associated cycle $c(\gamma)$}
 \label{fig:SurfaceProjection}
\end{figure}

In what follows, we will construct a hyperbolic metric on $\S$ and bound the length of the corresponding geodesic $\gamma$ by bounding a chosen representative: By cutting $c(\gamma)$ into pieces and using the combinatorics of the result to provide instructions for the construction of a curve homotopic to $c(\gamma)$, we will bound from above the length of $\gamma$ in terms of its self-intersection number.

\begin{figure}[h!]
 \centering 
 \includegraphics{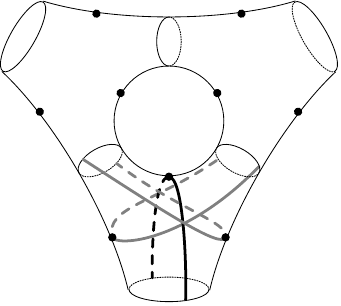}
 \caption{We construct a seam point on each seam edge.}
 \label{fig:SeamPoints}
\end{figure}

\subsection{Definition of $\tau$- and $\beta$- arcs and twisting numbers} 
When a convenient hyperbolic structure on $\S$ is chosen, the cycle $c(\gamma)$ can be broken into more manageable pieces. In each pair of pants, there are three simple geodesics that connect a boundary component to itself. Each of these three geodesics intersects a unique seam edge orthogonally, and we refer below to this intersection point as a \emph{seam point} (Figure \ref{fig:SeamPoints}). The seam points cut the cycle $c(\gamma)$ into arcs. When such an arc crosses a pants cuff, we call it a $\beta$-arc; otherwise it will be called a $\tau$-arc.  Note that $c(\gamma)$ is given by the union of the $\tau$- and $\beta$- arcs. See Figure \ref{fig:TauAndBeta} for an illustration.

\begin{figure}[h!]
 \centering 
 \includegraphics{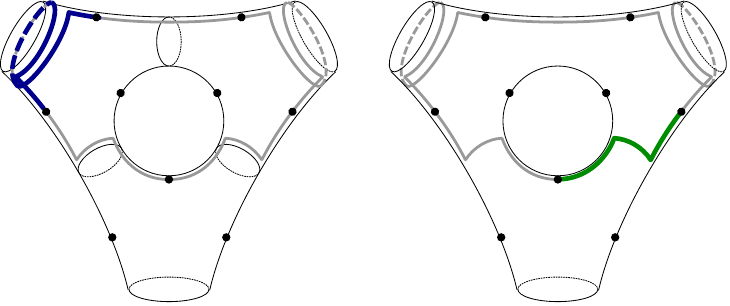}
 \caption{A $\tau$-arc is on the left and a $\beta$-arc is on the right.}
 \label{fig:TauAndBeta}
\end{figure}

In this description, each $\tau$- and $\beta$- arc acquires a twisting number. The twisting number $t_i$ of $\tau_i$ is a positive integer, and the twisting number $b_i$ of $\beta_i$ can be any integer. 
The twisting number is defined as follows: each $\tau$- and $\beta$- arc has a subarc $\eta$, contained in a pants cuff, given by the concatenation of some number of edges of the graph $G$. If $\eta$ is part of a $\tau$-arc $\tau_i$, then $t_i$ is the number of edges in $\eta$. For instance, the twisting number of the $\tau$-arc in Figure \ref{fig:TauAndBeta} is three. 

If $\eta$ is part of a $\beta$-arc $\beta_i$, then the number of edges in $\eta$ is the absolute value of $b_i$. The sign of $b_i$ is determined by a choice of orientation of $\S$; the sign of $b_i$ is given by the sense with which $\beta_i$ winds around the corresponding pants cuff. For example, the twisting number of the $\beta$-arc in  Figure \ref{fig:TauAndBeta} is either +1 or -1, depending on the orientation of $\S$.

We note that it is possible to determine precisely the length of a $\tau$- or $\beta$- arc from the length of the pants curve and the twisting number, but coarser estimates suffice in the proof of Theorem \ref{thm: shortest metric} (see Lemma \ref{length estimate lemma}). This will allow us to estimate the length of $c(\gamma)$ in terms of the twisting numbers of the $\tau$- and $\beta$- arcs. In order to relate the length of $\gamma$ to its self-intersection and the lengths of the pants curves, we will require a method of estimating twist numbers as a function of self-intersection.

\section{Inequalities relating self-intersection and twisting numbers} 
\label{sec: inequalities}

In \cite{Sapir}, a sequence of Dehn twists about curves in $\p$ are applied to $\S$ so that the twisting numbers associated to $c(\gamma)$ about each cuff $\alpha$ in $\p$ are related to $k$. Fix a cuff $\alpha$ of $\p$. Let the $\beta$-arcs of $\alpha$ be given by $\beta_\alpha =\{\beta_1,\ldots,\beta_n\}$ and let the $\tau$-arcs be given by $\tau_\alpha=\{\tau_1,\ldots,\tau_m\}$. The twisting coordinates are given by $b_1,\ldots,b_n\in \Z$ and $t_1,\ldots,t_m \in \Z_{>0}$, and by relabelling if necessary, we assume that $b_1\le b_2 \le \ldots \le b_n$ and $t_1\ge t_2 \ge \ldots \ge t_m$. Note that Lemma \ref{Jenya's pants} implies that $n\le D\sqrt{k}$. The following inequalities are obtained in \cite{Sapir}:

\begin{lem}\label{self-intersection inequalities}

There is a mapping class $\phi$, which is an appropriate product of powers of Dehn twists about the pants curves, and a constant $C$ so that the $\tau$- and $\beta$- arcs of $c(\phi \cdot \gamma)$ about each cuff $\alpha$ satisfy the following relationships:

\begin{itemize}
 \item The total number of $\tau$- and $\beta$- arcs is bounded by a multiple of $\sqrt k$: 
 \[
  n + m \leq C \cdot  \sqrt k
 \]
 
 \item When $n\ne 0$ the sum of the $b_i$'s satisfies:
 
 \[ 0 < \sum b_i \leq C \cdot \sqrt k 
 \]
 
 \item If we only consider those $\tau$-arcs with twisting number at least 4, we get
 \[
  \sum_{t_i \geq 4} i t_i \leq C \cdot k
 \]
 
 \item Regarding twisting numbers of $\beta$-arcs we have:
  \[
  \sum_{i > j}  b_i - b_j \leq C \cdot k
 \]

 \item Lastly, we have:
 \[
  n \sum t_i \leq C \cdot k
 \]
 
\end{itemize}

\noindent Note that the constant $C$ depends only on the topology of $\S$. 

\end{lem} 
 The second bullet is a consequence of the fact that a Dehn twist about $\alpha$ changes $\sum b_i$ by a factor of $2n$ and Lemma~\ref{Jenya's pants} tells us that $n \leq D \sqrt{k}$.

The last three bullet points of Lemma \ref{self-intersection inequalities} can be viewed as the contribution to the self-intersection number of $\gamma$ from the ``intersection" between different $\tau$-arcs, different $\beta$-arcs, and between $\tau$- and $\beta$- arcs, respectively, though complications arise due to the ambiguity inherent to the cycle $c(\gamma)$. We note that the fixed pants decomposition $\mathcal{P}$ is well suited for the curve $\phi \cdot \gamma$ in the sense that the estimate in Lemma \ref{Jenya's pants} still holds. 

The relevant facts from \cite{Sapir} for the proof of Lemma \ref{self-intersection inequalities} are contained in Proposition 3.1 and its proof. Specifially, we direct the reader to Lemma 3.6, Claims 4.7 and 4.9, Remarks 5.9 and 5.11 and inequality (4.3.3) of that paper. We warn the reader that there are small differences between the notation used here and in \cite{Sapir}. In particular, in that work each of the $\tau$- and $\beta$- arcs is contained in a pants cuff, whereas our $\tau$- and $\beta$- arcs include arcs from seam points to pants cuffs. Note that none of the inequalities in Lemma \ref{self-intersection inequalities} are affected by this discrepancy, and the differences in terminology will no longer be relevant. We use our notation for the sake of simplicity in the proof of Theorem \ref{thm: shortest metric}.

We need one additional estimate not contained in \cite{Sapir}.

\begin{lem}
\label{choosing twisting}
Suppose that $n\ne 0$. 
With $\phi$ equal to the product of powers of Dehn twists about the pants curves found in Lemma \ref{self-intersection inequalities}, the twist coordinates satisfy
\[
\sum_i |b_i| \le \frac{2C \cdot k}{n} + C\sqrt{k}~,
\]
where $C$ is the constant from Lemma \ref{self-intersection inequalities} depending only on the topology of $\S$.
\end{lem}

\begin{proof}

Note that all of the inequalities from Lemma \ref{self-intersection inequalities} hold. In particular, we have $0 < \sum b_i  \leq C \sqrt{k}$. 

We decompose our sum as $$\sum_i |b_i| = \sum_{b_i \leq0} |b_i| + \sum_{b_i > 0} b_i~,$$

\noindent and first show that $\sum_{b_i \leq 0} |b_i| \leq \frac{C \cdot k}{n}$. To simplify notation we let $$B_i = \sum_{j=1}^{i} b_j~.$$ We define a convex polygon $P$ in ${\bf R}^2$ whose vertices are $(i, B_i)$ and $(n- i, B_n - B_i)$  for $1 \leq i \leq n$ (see Figure \ref{fig:polygon} below).

\begin{figure}[h!]
\begin{overpic}[trim = 1.15in 2in 1.15in 1.75in, clip=true, totalheight=0.27\textheight]{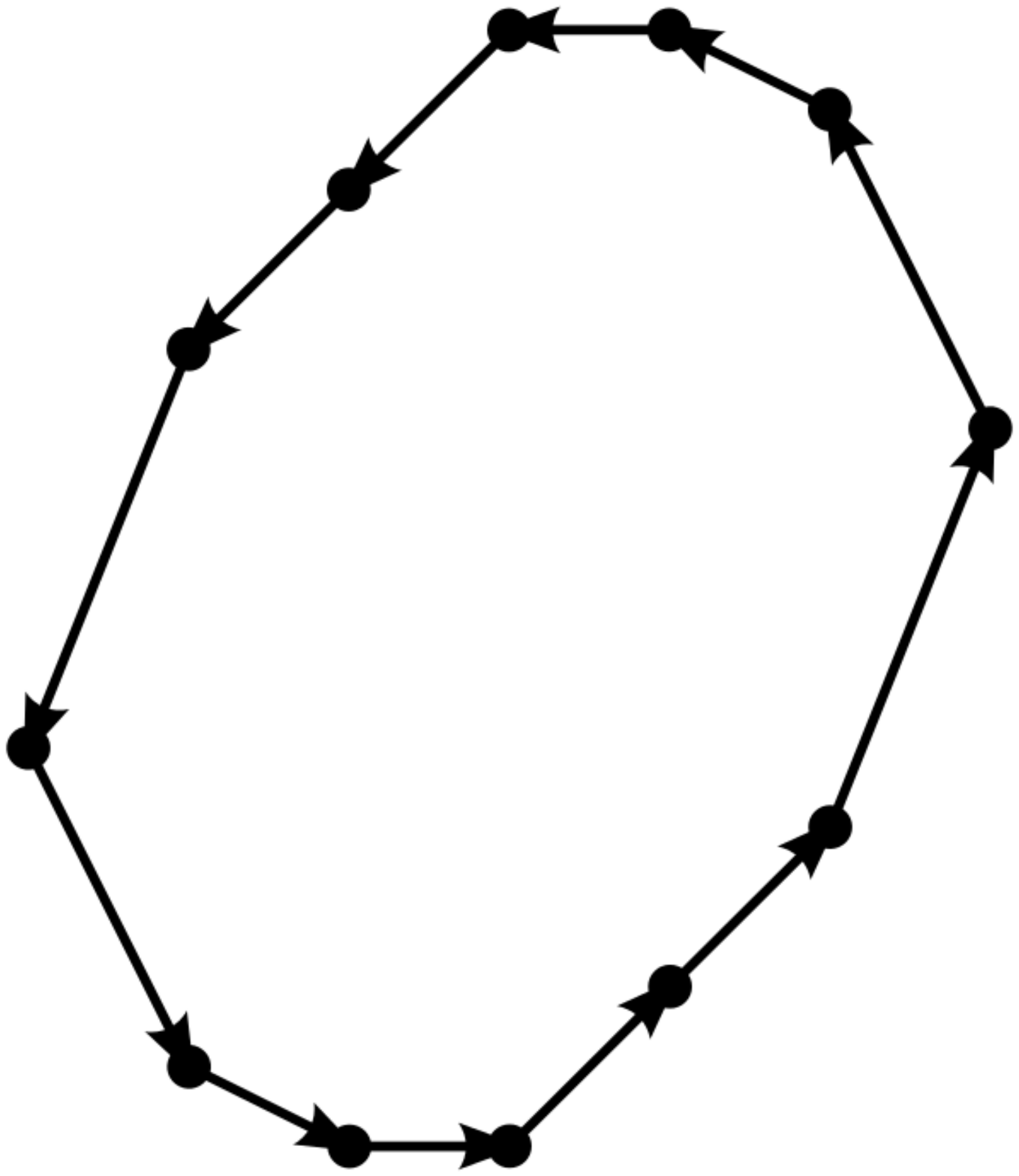}
\put(-14,36){\scriptsize{$(0, 0)$}}
\put(-3,8){\scriptsize{$(1, B_1)$}}
\put(15,-3){\scriptsize{$(2, B_2)$}}
\put(38,-3){\scriptsize{$(3, B_3)$}}
\put(58,13){\scriptsize{$(4, B_4)$}}
\put(72,27){\scriptsize{$(5, B_5)$}}
\put(85,60){\scriptsize{$(6, B_6)$}}
\put(75,90){\scriptsize{$(5, B_6 - B_1)$}}
\put(55,102){\scriptsize{$(4, B_6 - B_2)$}}
\put(13,100){\scriptsize{$(3, B_6 - B_3)$}}
\put(0,85){\scriptsize{$(2, B_6 - B_4)$}}
\put(-15,70){\scriptsize{$(1, B_6 - B_5)$}}
\end{overpic}
\caption{ }\label{fig:polygon}
\end{figure}

We remark that the idea for encoding the combinatorics of a collection of arcs and curves in a convex polygon in this fashion is due to Ser-Wei Fu, and the authors learned this technique during the course of several conversations with him. Similar techniques have been used in \cite{Tang-Webb}. We now calculate the area of $P$ using the triangles $T_i$ contained within $P$ and defined as follows: $T_i$, for $1 \leq i \leq n-1$, is the triangle with vertices $(0,0)$, $(i, B_i)$, and $(i+1, B_{i+1})$ (see Figure \ref{fig:polygon_with_tris}). 

\begin{figure}[h!]
\begin{overpic}[trim = 1.15in 2in 1.15in 1.75in, clip=true, totalheight=0.27\textheight]{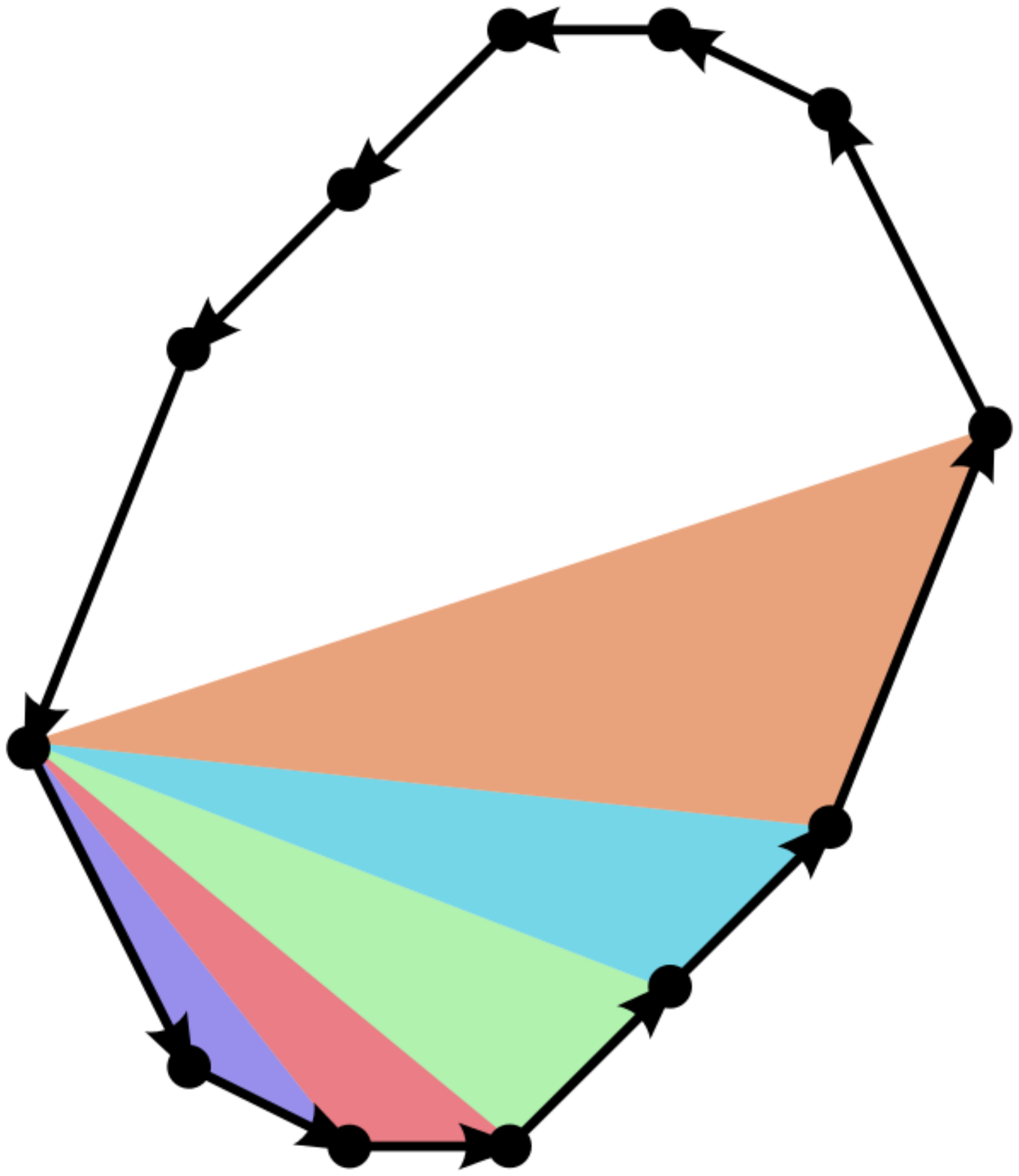}
\put(17.5,10){\scriptsize{$T_1$}}
\put(27,9){\scriptsize{$T_2$}}
\put(38.5,12){\scriptsize{$T_3$}}
\put(50,23){\scriptsize{$T_4$}}
\put(60,41){\scriptsize{$T_5$}}
\end{overpic}
\caption{ }\label{fig:polygon_with_tris}
\end{figure}

Thus, $$\sum_{i =1}^{n-1} Area(T_i) = \frac{1}{2} Area(P)~.$$ A calculation shows that $Area(T_i) = \frac{1}{2}[(b_{i+1} - b_1) + (b_{i+1} - b_2) + \cdots + (b_{i+1} - b_i)]$. Summing over all $1 \leq i \leq n-1$ we see that 

$$ Area (P) = 2 \sum_{i =1}^{n-1} Area(T_i) = \sum_{i > j} b_i - b_j~.$$

By Lemma \ref{self-intersection inequalities}, we have

\begin{equation}\label{eq:area_P}
Area (P) = \sum_{i > j} b_i - b_j \leq C \cdot k~.
\end{equation}

Using the four triangles $\mathfrak{T}_1, \mathfrak{T}_2, \mathfrak{T}_3$, and $\mathfrak{T}_4,$ in Figure \ref{fig:polygon_area} below, each of whose height is at least $\sum_{b_i \leq 0} |b_i|$. We conclude 

\begin{align}
Area(P) & \geq [Area(\mathfrak{T}_1) + Area(\mathfrak{T}_4)] + [Area(\mathfrak{T}_2) + Area(\mathfrak{T}_3)] \nonumber \\
& \geq  2 \left(\frac{1}{2} \, t \sum_{b_i \leq 0} |b_i| \right)+ 2\left( \frac{1}{2} (n-t) \sum_{b_i \leq 0} |b_i| \right) = n \sum_{b_i \leq 0} |b_i|~, \nonumber
\end{align}
where $t = \#\{b_i \; | \; b_i\le 0\}$.

\begin{figure}[h!]
\begin{overpic}[trim = 1.15in 2in 1.15in 1.75in, clip=true, totalheight=0.27\textheight]{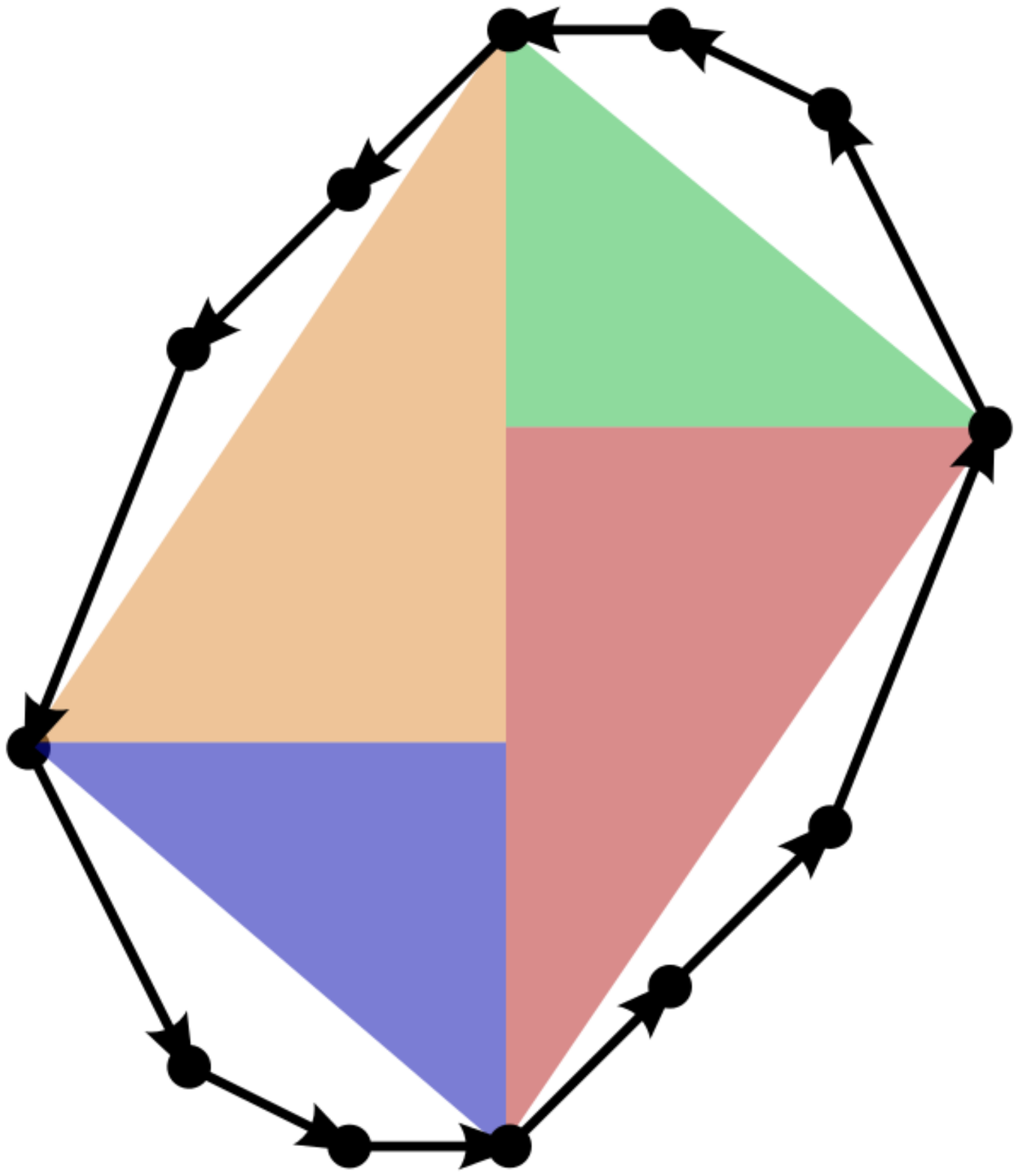}
\put(28,25){\large{$\mathfrak{T}_1$}}
\put(54,42){\large{$\mathfrak{T}_2$}}
\put(54,74){\large{$\mathfrak{T}_3$}}
\put(28,55){\large{$\mathfrak{T}_4$}}
\end{overpic}
\caption{ }\label{fig:polygon_area}
\end{figure}

Combining this lower bound on $Area(P)$ with Equation \ref{eq:area_P}, we have

$$n \sum_{b_i \leq 0} |b_i| \leq Area(P) \leq  C \cdot k~,$$ 

\noindent giving the desired bound on $\sum_{b_i \leq 0 } |b_i|$. We can rewrite this bound as $-\frac{C \cdot k}{n} \leq \sum_{b_i<0} b_i$  to obtain 

$$ -\frac{C \cdot k}{n} + \sum_{b_i \geq 0} b_i \leq \sum_{b_i<0} b_i +  \sum_{b_i \geq 0} b_i \leq C\sqrt{k}~,$$
by Lemma \ref{self-intersection inequalities}.

Thus, $\sum_{b_i \geq 0} b_i \leq \frac{C \cdot k}{n} + C\sqrt{k} $ and $\sum_i |b_i| \leq \frac{2C \cdot k}{n} + C\sqrt{k}$.

\end{proof}

\section{Proof of Theorem \ref{thm: shortest metric}}\label{sec:short geodesics}
Note that $\rho$ is a metric in which $\ell_\rho(\gamma) \leq C_3 \sqrt k$ if and only if $\ell_{\phi^{-1} \cdot \rho}(\phi \cdot \gamma) \leq C_3 \sqrt k$, so it is enough to find a metric in which the length of $\phi \cdot \gamma$ is bounded. For notational convenience, in what follows we replace $\phi \cdot \gamma$ with $\gamma$, so that in particular we now have that the inequalities from Section \ref{sec: inequalities} apply verbatim to $\gamma$. We will pick a hyperbolic metric on $\S$ by choosing appropriate lengths for the pants curves, and we bound the length of $c(\gamma)$ in that metric using the twist parameters about each $\alpha \in \p$. The inequalities in Lemmas \ref{self-intersection inequalities} and \ref{choosing twisting} will allow us to finish the proof.

We record some evident hyperbolic geometry facts to aid in our analysis. The curve formed by a component of the boundary of the $L$-neighborhood of a geodesic of length $\ell$ has length $\ell\cosh (L)$. We will call such a curve an \emph{equidistant curve}.

\begin{lem}
\label{hyperbolic exercise}
Suppose that the curves in the boundary of a pair of pants have length at most $C$. 
Then the distance from a seam point to a neighboring equidistant curve of length $C$ is at most $D$, where $D$ is a constant depending only on $C$. 
\end{lem}

\begin{proof}
The given distance is a continuous function of the lengths of the three geodesic boundary components that extends continuously to $[0,C]^3$. 
\end{proof}

\begin{lem}
\label{length estimate lemma}
Suppose that the length of $\alpha$ is $\epsilon$, and that $\epsilon\le 1$. Then the lengths of the $\beta$-arcs satisfy:
\begin{align}
\ell(\beta_\alpha) & \le \frac{\epsilon}{2} \sum_i |b_i| \ + \ 2n \log\left(\frac{2}{\epsilon} \right) +O(n)
\end{align}

For the $\tau$-arcs we have:
\begin{align} 
\ell(\tau_\alpha) & \le \frac{\epsilon}{2} \sum_j t_j + 2 \sum_j \log (t_j) +O(m)~.
\end{align}
\end{lem}
\begin{proof}
The length of an arc traveling from the equidistant curve of length $1$ to the equidistant curve of length $T$ is $$\cosh^{-1}(1/\epsilon) - \cosh^{-1}(T/\epsilon)~.$$ Note that the existence of such an arc requires that $1\ge T\ge \epsilon$. It is straightforward to check that $\cosh^{-1}(1/\epsilon) - \cosh^{-1}(T/\epsilon) \le \log(2/T)$ for $1\ge T \ge\epsilon$.  

The $\beta$-arc $\beta_i$ (when oriented) travels from the seam point on a seam edge neighboring $\alpha$ to the equidistant curve of length $1$, from the equidistant curve of length $1$ to the geodesic representative of $\alpha$, traverses $|b_i|$ edges of $G$, from the geodesic representative of $\alpha$ to the equidistant curve of length $1$, and finally from the equidistant curve of length $1$ to another seam point. The first and last leg of this journey can be done in $O(1)$ length by Lemma \ref{hyperbolic exercise}, and the remainder has length less than or equal to $2\log (2/\epsilon)+\epsilon|b_i|/2$. The length bound for $\beta_\alpha$ follows.

An oriented $\tau$-arc $\tau_j$ similarly travels from a seam point to the equidistant curve of length $1$, from the equidistant curve of length $1$ to the geodesic representative of $\alpha$, traverses $t_j$ edges of $G$, and finishes similarly to the way it came. This representative has length given by $2\log(2/\epsilon)+\epsilon t_j/2+O(1)$.

If $t_j/2 \le 1/\epsilon$ for some $j$, then there is a more efficient representative of the arc $\tau_j$, homotopic to the arc $\tau_j$ relative to its seam endpoints, but which manages to perform the `twisting' of $\tau_j$ around $\alpha$ with length contribution $O(1)$. More precisely, consider the representative of $\tau_j$ that travels from the starting seam point to the equidistant curve of length $1$, from the equidistant curve of length $1$ to the equidistant curve of length $2/t_j$ --- note that $1\ge 2/t_j \ge \epsilon$ as needed --- and finishes along a seam edge of $G$ as it came. This representative has length bounded above by $2\log(t_j) +1+ O(1)$.

On the other hand, if $t_j/2\ge 1/\epsilon$ then $2\log(2/\epsilon)\le 2\log(t_j)$ and we conclude
\begin{align*}
\ell(\tau_j) \le \left\{ 
	\begin{array}{ll}
		2\log (t_j) + 1+ O(1) \ , & \text{ if } t_j \le 2/\epsilon \\
\displaystyle		2\log(t_j) + \frac{\epsilon t_j}{2} + O(1) \ , & \text{ if } t_j \ge 2/\epsilon~.
	\end{array}
\right.
\end{align*}
Thus we have $$\ell(\tau_\alpha) \le \sum_j \max\left\{1, \frac{\epsilon t_j}{2} \right\} + 2 \sum_j \log (t_j) +O(m)~.$$
The claimed length estimate for $\tau_\alpha$ follows.
\end{proof}
 
\begin{lem}
\label{sums to products}
If $x_1,\ldots,x_n\in\R$ satisfy $\displaystyle \sum_j jx_j \le K$, then 
$ \log \displaystyle \prod_j x_j \le 2\sqrt{ K/e }.$
\end{lem}

\begin{proof}
Let $f(x_1,\ldots,x_n)=\log \displaystyle \prod_j x_j$ and $g(x_1,\ldots,x_n)=\displaystyle \sum_j jx_j$. The method of Lagrange multipliers implies that the critical points of $f$ subject to the condition that $g(x_1,\ldots,x_n)=K$ satisfy $$\lambda \nabla f = \nabla g$$ for some $\lambda \in \R$. This implies that $x_j=\lambda/j$ for each $j=1,\ldots,n$ and the restriction on $g$ implies that $\lambda=K/n$. This yields the value $$f(x_1,\ldots,x_n)=\log \frac{K^n}{n^nn!}~,$$easily seen to be a maximum for $f$ restricted to $g=K$. We use the estimate $n!\ge (n/e)^n$ and optimize using calculus of one variable: $f(x_1,\ldots,x_n) \le n \log (eK/n^2)$, and setting the derivative of this last expression equal to zero gives $n=\sqrt{K/e}.$ This implies that $f(x_1,\ldots,x_n)\le 2\sqrt{K/e}$ for all $n$, as desired.
\end{proof}

We are now ready for the proof of Theorem \ref{thm: shortest metric}. 

\begin{proof}[Proof of Theorem \ref{thm: shortest metric}]
Recall that $m$ and $n$ are the numbers of $\tau$- and $\beta$- arcs, respectively. Suppose first that $n\ne 0$. Choosing any convenient hyperbolic metric on $\S$ so that the length of each pants curve is at most $1$, and so that the length of $\alpha$ is $\epsilon \le 1$, Lemmas \ref{length estimate lemma} and \ref{choosing twisting} imply that we have 
\begin{align*}
\ell(\tau_\alpha \cup \beta_\alpha) & \le \frac{\epsilon}{2} \sum_i |b_i| \ + \ 2n \log\left(\frac{2}{\epsilon} \right) +O(n) + \frac{\epsilon}{2} \sum_j t_j + 2 \sum_j \log t_j +O(m) \\
& \le C \frac{\epsilon \ k}{n} + \frac{C}{2}\epsilon \sqrt{k} + 2 n\log \left(\frac{2}{\epsilon}\right) + O(n) + \frac{\epsilon}{2} \sum_j t_j + 2\sum_j \log t_j +O(m)~.
\end{align*}
The theorem will follow from making a choice of $\epsilon$ that ensures each of these terms is $O(\sqrt{k})$. We refine our choice $$ \epsilon = \min\left\{ \frac{n}{\sqrt{k}} , 1\right\}$$ for the length of $\alpha$. We check the relevant inequalities:\\

\begin{enumerate}

\item If $n\ge \sqrt{k}$ then $k/n \le \sqrt{k}$ and $\epsilon =1$, so that $\epsilon k/n$ is $O(\sqrt{k})$; whereas if $n\le \sqrt{k}$ then $\epsilon =n/\sqrt{k}$ and $\epsilon k/n$ is $O(\sqrt{k})$. In either case, $\epsilon k /n$ is $O(\sqrt{k})$.
\item Evidently $\epsilon \sqrt{k} \le n$, which is $O(\sqrt{k})$ by Lemma \ref{Jenya's pants}. 
\item A calculus exercise shows that $n\log(2\sqrt{k}/n) $ is $O(\sqrt{k})$.
\item We have $n$ is $O(\sqrt{k})$ by Lemma \ref{Jenya's pants}.
\item When $n\ge \sqrt{k}$ then $k/n \le \sqrt{k}$ and $\epsilon =1$, so that by the last bullet point of Lemma~\ref{Jenya's pants} we see $\displaystyle \frac{\epsilon}{2} \sum_j t_j \le \frac{C \cdot k}{2n}$, which is $O(\sqrt{k})$. Otherwise, we have $\displaystyle \frac{\epsilon}{2} \sum_j t_j \le \frac{n}{\sqrt{k}} \cdot \frac{C \cdot k}{n}$, which is $O(\sqrt{k})$.
\item Using Lemma \ref{self-intersection inequalities} again we have $\displaystyle \sum_j j t_j = \sum_{t_j < 4} jt_j + \sum_{t_j \geq 4} jt_j  \le 3m^2+ C \cdot k \leq C' \cdot k$, and Lemma \ref{sums to products} implies that $\displaystyle \sum_j\log t_j$ is $O(\sqrt{k})$.
\item Lemma \ref{self-intersection inequalities} implies that $m$ is $O(\sqrt{k})$.
\end{enumerate}
Note that the above choice of $\epsilon$ satisfies $\epsilon \ge 1/\sqrt{k}$.

If, on the other hand, $n=0$, then we make the choice $\epsilon = 1/\sqrt{k}$. By Lemma \ref{length estimate lemma} we have 
\begin{align*}  
\ell(\tau_\alpha)  \le \frac{\epsilon}{2} \sum_j t_j + 2 \sum_j \log (t_j) +O(m) ~.
\end{align*}
 Note that $\sum t_j$ is $O(k)$ as in (6) above, so that $\epsilon \sum t_j$ is $O(\sqrt{k})$. The other terms follow as in (6) and (7) above. 

We conclude that $\ell(\beta_\alpha \cup \tau_\alpha)$ is $O(\sqrt{k})$, with constants depending only on the topology of $\S$, and that the geodesic pants cuff $\alpha$ has length at least $1/\sqrt{k}$. Since $\gamma$ is the union of $\tau$- and $\beta$- arcs, we apply the above analysis one-by-one at each of the pants cuffs. The desired bound on $\ell(\gamma)$ follows.

As for the lower bound on the injectivity radius, recall that the collar lemma states that a simple closed curve of length $l$ has an embedded tubular neighborhood of width $\log(\coth(l/2))$. If $c$ is any simple closed geodesic, it is either a cuff of $\p$, or it must intersect at least one cuff of $\p$; thus either the desired lower bound follows immediately or $c$ has length at least $\log(\coth(1/2\sqrt{k}))$, which is larger than $1/\sqrt{k}$ for all $k\geq 2$.  When $k=1$, $\log(\coth(1/2)) > 1/2$; the theorem follows.

Finally, we address the case where $\S$ has cusps or non-compact ends. In the case of non-compact flaring ends, simply cut off the open ends to obtain a compact surface with totally geodesic boundary in place of each flaring end. Let $\S(\gamma)$ denote the compact hyperbolic surface for which $\gamma$ admits the desired short representative, and suppose $\S$ has $h$ boundary components, labeled $b_{1},...,b_{h}$. Then given $\epsilon < 1/\sqrt{k}$ and $A \subseteq \left\{1,..., h\right\}$, let $\S(\gamma)_{\epsilon, A}$ denote the hyperbolic surface for which each non-peripheral cuff of $\p$ has the same length and twisting coordinate as on the original surface $\S(\gamma)$, but so that we replace each of the boundary components whose indices lie within $A$, with one of length $\epsilon$. Then we claim that $\gamma$ admits a representative satisfying the desired length bound on $\S(\gamma)_{\epsilon, A}$ as well. Indeed, if $\alpha$ is a peripheral cuff of $\p$ on $\S(\gamma)_{\epsilon, A}$, the contribution to $\ell(\gamma)$ from arcs entering the (appropriately chosen) tubular neighborhood of $\alpha$ are all $\tau$ arcs. Thus, as above denoting these arcs as $\tau_{\alpha}$ we have
\[ \ell(\tau_\alpha) \le \frac{\epsilon}{2} \sum_j t_j + 2 \sum_j \log (t_j) +O(m). \]   
Hence, the first summand is smaller than the corresponding one on $\S(\gamma)$, and the second two summands are unaffected by the change. Thus, given $\delta>0$ small, choose $\epsilon$ sufficiently small so that $\ell_{\S(\gamma)_{\epsilon}}(\gamma)$ is within $\delta$ of the surface obtained by replacing with a cusp each of the boundary components with indices in $A$. Letting $\delta \rightarrow 0$, the theorem follows for cusped surfaces as well. 

\end{proof}

\subsection{Questions and further remarks}
\label{sec:Questions and remarks}

Given a particular curve (or family of curves), the construction in our proof of Theorem \ref{thm: lifts simply in n} often grants enough flexibility to strengthen the conclusion. We elaborate:

\begin{rem} \label{thick}  Given any $\epsilon >0$, there exists $N = N(\epsilon, \S)$ so that for any curve $\gamma$ with $i(\gamma, \gamma)\leq k$ whose length is minimized in the $\epsilon$-thick part of Teichm{\"u}ller space, there exists a cover of degree $N \cdot \sqrt{k}$ to which $\gamma$ lifts simply. 
\end{rem}

\begin{rem}
 \label{rem:Thick part condition}
The proof of Theorem \ref{thm: shortest metric} gives a sufficient condition for $\gamma$ to have short length (i.e.~on the order of $\sqrt{k}$) in the thick part of Teichmuller space, in terms of the twisting numbers of its $\tau$- and $\beta$- arcs:

Suppose that $c(\gamma)$ has $\tau$- and $\beta$- arcs with twisting numbers $t_1, \dots, t_m$ and $b_1, \dots, b_m$, respectively, and that there exists an $O(k)$ function $f$ so that 
 \[
  \sum t_i + \sum |b_i| \leq f(k) ~.
 \]
Note that such a function always exists by Lemmas \ref{self-intersection inequalities} and \ref{choosing twisting}. Then the proof of Theorem \ref{thm: shortest metric} produces a metric $\rho$ where 
 \[
  \ell_\rho(\gamma) \leq  C_3 \sqrt k
 \]
and so that the systole length of $(\S,\rho)$ is $O(\sqrt{k}/f(k))$. As above, we have $\deg(\gamma)$ is $O(f(k))$. In particular, if $f$ is $O(\sqrt{k})$ then $\deg(\gamma)$ is $O(\sqrt{k})$.
\end{rem}

\begin{ques}
Can an upper bound for the degree $\deg(\gamma)$ be made explicit as a function of the twisting numbers of the $\tau$- and $\beta$- arcs?
\end{ques}

\begin{rem} \label{liouville} One method to produce curves to which Remark \ref{thick} applies is to use `random geodesics' arising in Bonahon's theory of geodesic currents \cite{Bonahon}. Suppose $\S$ is compact, and fix a hyperbolic metric $\rho$ on $\S$. Apply the geodesic flow for time $T$ to a random unit tangent vector, and close up the curve arbitrarily with an arc of length bounded by the diameter of the surface. Bonahon's theory implies that these curves projectively limit to the Liouville current of $(\S,\rho)$, whose length function is minimized at $\rho \in \T(\S)$. Remark \ref{thick} implies that $\deg(\gamma_T)$ is growing at most as the square root of the self intersection. (Alternatively, note that by \cite[Theorem 1]{lalley}, $\gamma_T$ is almost surely of self-intersection roughly $T^2$ and length roughly $T$, and the Lipschitz map from $(\S,\rho)$ to $(\S,\rho_0)$ implies that $\deg(\gamma_T)$ is growing at most as the square root of the self-intersection number).
\end{rem}

\begin{ques} \label{random degree} What can be said about lower bounds for $\deg(\gamma_T)$? Is there almost surely linear growth in $T$? 
\end{ques}

\begin{ques} \label{stats of infima} What can be said about the statistics of the finitely many infima of length functions of curves with self-intersection $k$? By \cite[Cor.~1.4]{Basmajian} and Corollary \ref{max infimum}, each such infimum is roughly between $\log k$ and $\sqrt{k}$. What is the average?
\end{ques}

\textbf{Acknowledgements}. The authors thank Martin Bridgeman, Chris Leininger, Feng Luo, Kasra Rafi, and Tian Yang for helpful conversations. The first author was fully supported by NSF postdoctoral fellowship grant DMS-1502623.


\end{document}